\documentclass[12pt]{article}
\usepackage{amsmath,amssymb,amsthm}
\numberwithin{equation}{section}
\usepackage{units}
\usepackage{color}
\usepackage[T1]{fontenc}
\usepackage[utf8]{inputenc}
\usepackage{authblk}
\usepackage{bm}
\usepackage{graphicx}

\usepackage{datetime}

\usepackage[colorlinks=true,
linkcolor=webgreen,
filecolor=webbrown,
citecolor=webgreen]{hyperref}

\definecolor{webgreen}{rgb}{0,.5,0}
\definecolor{webbrown}{rgb}{.6,0,0}

\usepackage[toc,page]{appendix}

\newcommand{\Z}{{\mathbb Z}}

\newtheorem{thm}{Theorem}
\newtheorem{theorem}[thm]{Theorem}

\newtheorem{corollary}[thm]{Corollary}

\title{New Tribonacci Recurrence Relations and Addition Formulas}

\author[]{Kunle Adegoke\thanks{Corresponding author: \href{mailto:adegoke00@gmail.com}{\tt adegoke00@gmail.com}}}
\author{Adenike Olatinwo}
\author{Winning Oyekanmi}
\affil{Department of Physics and Engineering Physics, \mbox{Obafemi Awolowo University}, 220005 Ile-Ife, Nigeria}

\begin{document}

\date{}

\maketitle

\begin{abstract}
\noindent Only one three-term recurrence relation, namely, $W_{r}=2W_{r-1}-W_{r-4}$, is known for the generalized Tribonacci numbers, $W_r$, $r\in\Z$, defined by $W_{r}=W_{r-1}+W_{r-2}+W_{r-3}$ and \mbox{$W_{-r}=W_{-r+3}-W_{-r+2}-W_{-r+1}$}, where $W_0$, $W_1$ and $W_2$ are given, arbitrary integers, not all zero. Also, only one four-term addition formula is known for these numbers, which is, $W_{r + s}  = T_{s - 1} W_{r - 1}  + (T_{s - 1}  + T_{s-2} )W_r  + T_s W_{r + 1}$, where $({T_r})_{r\in\Z}$ is the Tribonacci sequence, a special case of the generalized Tribonacci sequence, with $W_0=T_0=0$ and $W_1=W_2=T_1=T_2=1$. In this paper we discover three new three-term recurrence relations and two identities from which a plethora of new addition formulas for the generalized Tribonacci numbers may be discovered. We obtain a simple relation connecting the Tribonacci numbers and the Tribonacci-Lucas numbers. Finally, we derive quadratic and cubic recurrence relations for the generalized Tribonacci numbers.

\end{abstract}
\section{Introduction}
For $r\ge 3$, we define the generalized Tribonacci numbers $W_r$ by the third order recurrence relation:
\begin{equation}\label{eq.qqyiuit}
W_{r}=W_{r-1}+W_{r-2}+W_{r-3}\,,
\end{equation}
where $W_0$, $W_1$ and $W_2$ are arbitrary integers. By writing $W_{r-1}=W_{r-2}+W_{r-3}+W_{r-4}$ and subtracting this from relation \eqref{eq.qqyiuit}, we see that $W_r$ also obeys the useful three-term recurrence
\begin{equation}\label{eq.ccgxij2}
W_r=2W_{r-1}-W_{r-4}\,.
\end{equation}
Extension of the definition of the generalized Tribonacci numbers to negative subscripts is provided by writing identity \eqref{eq.ccgxij2} as $W_{r+4}=2W_{r+3}-W_{r}$; so that
\begin{equation}
W_{ - r}  = 2W_{ - r + 3}  - W_{ - r + 4}\,.
\end{equation}
Well known examples of $W_r$ are the Tribonacci sequence, $(T_r)$, $r\in\Z$, for which $W=T$, $W_0=T_0=0$, $W_1=W_2=T_1=T_2=1$ and the Tribonacci-Lucas sequence , $(K_r)$, $r\in\Z$, for which $W=K$, $W_0=K_0=3$, $W_1=K_1=1$, $W_2=K_2=3$. Table \ref{tab.ms2wgo8} shows the first few Tribonacci and Tribonacci-Lucas numbers for $-20\le r\le 24$.
\begin{table}[h!]
\begin{tabular}{lllllllllll}
\multicolumn{1}{c}{$r$} & \multicolumn{1}{c}{} & \multicolumn{1}{c}{$-20$} & \multicolumn{1}{c}{$-19$} & \multicolumn{1}{c}{$-18$} & \multicolumn{1}{c}{$-17$} & \multicolumn{1}{c}{$-16$} & \multicolumn{1}{c}{$-15$} & \multicolumn{1}{c}{$-14$} & \multicolumn{1}{c}{$-13$} & \multicolumn{1}{c}{$-12$} \\ 
\hline
\multicolumn{1}{c}{$T_r$} & \multicolumn{1}{c}{} & \multicolumn{1}{c}{$-56$} & \multicolumn{1}{c}{$159$} & \multicolumn{1}{c}{$-103$} & \multicolumn{1}{c}{$0$} & \multicolumn{1}{c}{$56$} & \multicolumn{1}{c}{$-47$} & \multicolumn{1}{c}{$9$} & \multicolumn{1}{c}{$18$} & \multicolumn{1}{c}{$-20$} \\ 
\multicolumn{1}{c}{$K_r$} & \multicolumn{1}{c}{} & \multicolumn{1}{c}{$795$} & \multicolumn{1}{c}{$-571$} & \multicolumn{1}{c}{$47$} & \multicolumn{1}{c}{$271$} & \multicolumn{1}{c}{$-253$} & \multicolumn{1}{c}{$65$} & \multicolumn{1}{c}{$83$} & \multicolumn{1}{c}{$-105$} & \multicolumn{1}{c}{$43$} \\ 
\multicolumn{1}{c}{} & \multicolumn{1}{c}{} & \multicolumn{1}{c}{} & \multicolumn{1}{c}{} & \multicolumn{1}{c}{} & \multicolumn{1}{c}{} & \multicolumn{1}{c}{} & \multicolumn{1}{c}{} & \multicolumn{1}{c}{} & \multicolumn{1}{c}{} & \multicolumn{1}{c}{} \\ 
\multicolumn{1}{c}{$r$} & \multicolumn{1}{c}{} & \multicolumn{1}{c}{$-11$} & \multicolumn{1}{c}{$-10$} & \multicolumn{1}{c}{$-9$} & \multicolumn{1}{c}{$-8$} & \multicolumn{1}{c}{$-7$} & \multicolumn{1}{c}{$-6$} & \multicolumn{1}{c}{$-5$} & \multicolumn{1}{c}{$-4$} & \multicolumn{1}{c}{$-3$} \\ 
\hline
\multicolumn{1}{c}{$T_r$} & \multicolumn{1}{c}{} & \multicolumn{1}{c}{$7$} & \multicolumn{1}{c}{$5$} & \multicolumn{1}{c}{$-8$} & \multicolumn{1}{c}{$4$} & \multicolumn{1}{c}{$1$} & \multicolumn{1}{c}{$-3$} & \multicolumn{1}{c}{$2$} & \multicolumn{1}{c}{$0$} & \multicolumn{1}{c}{$-1$} \\ 
\multicolumn{1}{c}{$K_r$} & \multicolumn{1}{c}{} & \multicolumn{1}{c}{$21$} & \multicolumn{1}{c}{$-41$} & \multicolumn{1}{c}{$23$} & \multicolumn{1}{c}{$3$} & \multicolumn{1}{c}{$-15$} & \multicolumn{1}{c}{$11$} & \multicolumn{1}{c}{$-1$} & \multicolumn{1}{c}{$-5$} & \multicolumn{1}{c}{$5$} \\ 
\multicolumn{1}{c}{} & \multicolumn{1}{c}{} & \multicolumn{1}{c}{} & \multicolumn{1}{c}{} & \multicolumn{1}{c}{} & \multicolumn{1}{c}{} & \multicolumn{1}{c}{} & \multicolumn{1}{c}{} & \multicolumn{1}{c}{} & \multicolumn{1}{c}{} & \multicolumn{1}{c}{} \\ 
\multicolumn{1}{c}{$r$} & \multicolumn{1}{c}{} & \multicolumn{1}{c}{$-2$} & \multicolumn{1}{c}{$-1$} & \multicolumn{1}{c}{$0$} & \multicolumn{1}{c}{$1$} & \multicolumn{1}{c}{$2$} & \multicolumn{1}{c}{$3$} & \multicolumn{1}{c}{$4$} & \multicolumn{1}{c}{$5$} & \multicolumn{1}{c}{$6$} \\ 
\hline
\multicolumn{1}{c}{$T_r$} & \multicolumn{1}{c}{} & \multicolumn{1}{c}{$1$} & \multicolumn{1}{c}{$0$} & \multicolumn{1}{c}{$0$} & \multicolumn{1}{c}{$1$} & \multicolumn{1}{c}{$1$} & \multicolumn{1}{c}{$2$} & \multicolumn{1}{c}{$4$} & \multicolumn{1}{c}{$7$} & \multicolumn{1}{c}{$13$} \\ 
\multicolumn{1}{c}{$K_r$} & \multicolumn{1}{c}{} & \multicolumn{1}{c}{$-1$} & \multicolumn{1}{c}{$-1$} & \multicolumn{1}{c}{$3$} & \multicolumn{1}{c}{$1$} & \multicolumn{1}{c}{$3$} & \multicolumn{1}{c}{$7$} & \multicolumn{1}{c}{$11$} & \multicolumn{1}{c}{$21$} & \multicolumn{1}{c}{$39$} \\ 
\multicolumn{1}{c}{} & \multicolumn{1}{c}{} & \multicolumn{1}{c}{} & \multicolumn{1}{c}{} & \multicolumn{1}{c}{} & \multicolumn{1}{c}{} & \multicolumn{1}{c}{} & \multicolumn{1}{c}{} & \multicolumn{1}{c}{} & \multicolumn{1}{c}{} & \multicolumn{1}{c}{} \\ 
\multicolumn{1}{c}{$r$} & \multicolumn{1}{c}{} & \multicolumn{1}{c}{$7$} & \multicolumn{1}{c}{$8$} & \multicolumn{1}{c}{$9$} & \multicolumn{1}{c}{$10$} & \multicolumn{1}{c}{$11$} & \multicolumn{1}{c}{$12$} & \multicolumn{1}{c}{$13$} & \multicolumn{1}{c}{$14$} & \multicolumn{1}{c}{$15$} \\ 
\hline
\multicolumn{1}{c}{$T_r$} & \multicolumn{1}{c}{} & \multicolumn{1}{c}{$24$} & \multicolumn{1}{c}{$44$} & \multicolumn{1}{c}{$81$} & \multicolumn{1}{c}{$149$} & \multicolumn{1}{c}{$274$} & \multicolumn{1}{c}{$504$} & \multicolumn{1}{c}{$927$} & \multicolumn{1}{c}{$1705$} & \multicolumn{1}{c}{$3136$} \\ 
\multicolumn{1}{c}{$K_r$} & \multicolumn{1}{c}{} & \multicolumn{1}{c}{$71$} & \multicolumn{1}{c}{$131$} & \multicolumn{1}{c}{$241$} & \multicolumn{1}{c}{$443$} & \multicolumn{1}{c}{$815$} & \multicolumn{1}{c}{$1499$} & \multicolumn{1}{c}{$2757$} & \multicolumn{1}{c}{$5071$} & \multicolumn{1}{c}{$9327$} \\ 
\multicolumn{1}{c}{} & \multicolumn{1}{c}{} & \multicolumn{1}{c}{} & \multicolumn{1}{c}{} & \multicolumn{1}{c}{} & \multicolumn{1}{c}{} & \multicolumn{1}{c}{} & \multicolumn{1}{c}{} & \multicolumn{1}{c}{} & \multicolumn{1}{c}{} & \multicolumn{1}{c}{} \\ 
\multicolumn{1}{c}{$r$} & \multicolumn{1}{c}{} & \multicolumn{1}{c}{$16$} & \multicolumn{1}{c}{$17$} & \multicolumn{1}{c}{$18$} & \multicolumn{1}{c}{$19$} & \multicolumn{1}{c}{$20$} & \multicolumn{1}{c}{$21$} & \multicolumn{1}{c}{$22$} & \multicolumn{1}{c}{$23$} & \multicolumn{1}{c}{$24$} \\ 
\hline
\multicolumn{1}{c}{$T_r$} & \multicolumn{1}{c}{} & \multicolumn{1}{c}{$5768$} & \multicolumn{1}{c}{$10609$} & \multicolumn{1}{c}{$19513$} & \multicolumn{1}{c}{$35890$} & \multicolumn{1}{c}{$66012$} & \multicolumn{1}{c}{$121415$} & \multicolumn{1}{c}{$223317$} & \multicolumn{1}{c}{$410744$} & \multicolumn{1}{c}{$755476$} \\ 
\multicolumn{1}{c}{$K_r$} & \multicolumn{1}{c}{} & \multicolumn{1}{c}{$17155$} & \multicolumn{1}{c}{$31553$} & \multicolumn{1}{c}{$58035$} & \multicolumn{1}{c}{$106743$} & \multicolumn{1}{c}{$196331$} & \multicolumn{1}{c}{$361109$} & \multicolumn{1}{c}{$664183$} & \multicolumn{1}{c}{$1221623$} & \multicolumn{1}{c}{$2246915$} \\ 
\end{tabular}
\caption{The first few Tribonacci and Tribonacci-Lucas numbers.}
\label{tab.ms2wgo8}
\end{table}

\medskip

The following references contain useful information on the properties of the Tribonacci numbers and related: \cite{adegoke18c,feng11,frontczak18, gabai70,shah11,waddill67}.

\medskip

Among other interesting results, we found the following three-term recurrence relations which are presumably new:
\[
W_{r-16}=-103W_r+56W_{r+1}\,,
\]
\[
2W_{r-17}=9W_r-103W_{r-4}
\]
and
\[
W_{r-14}=9W_{r+2}-56W_{r-1}\,.
\]
We also found a simple relation linking the Tribonacci numbers and the Tribonacci-Lucas numbers:
\[
K_{r-2}=5T_{r-1}-T_{r+1}\,.
\]
\section{Linear recurrence relations and addition formulas}
\begin{theorem}\label{thm.ash3b4f}
The following identity holds for integers $a$, $b$, $c$, $d$ and $e$:
\[
\begin{split}
&(T_{a-c}T_{c-b}T_{b-a}+T_{b-c}T_{c-a}T_{a-b})W_{d+e}\\
&=\quad(T_{b-c}T_{c-a}T_{e-b}-T_{b-c}T_{c-b}T_{e-a}+T_{c-b}T_{b-a}T_{e-c})W_{d+a}\\
&\qquad+(T_{a-c}T_{c-b}T_{e-a}-T_{a-c}T_{c-a}T_{e-b}+T_{c-a}T_{a-b}T_{e-c})W_{d+b}\\
&\quad\qquad+(T_{b-c}T_{a-b}T_{e-a}-T_{a-b}T_{b-a}T_{e-c}+T_{a-c}T_{b-a}T_{e-b})W_{d+c}
\end{split}
\]

\end{theorem}
\begin{proof}
We seek to express a generalized Tribonacci number as a linear combination of three Tribonacci numbers. Let 
\begin{equation}\label{eq.nbl33qg}
W_{d+e}=f_1T_{e-a}+f_2T_{e-b}+f_3T_{e-c}\,,
\end{equation}
where $a$, $b$, $c$, $d$ and $e$ are arbitrary integers and the coefficients $f_1$, $f_2$ and $f_3$ are to be determined. Setting $e=a$, $e=b$ and $e=c$, in turn, we obtain three simultaneous equations:
\[
W_{d+a}=f_2T_{a-b}+f_3T_{a-c}\,,\quad W_{d+b}=f_1T_{b-a}+f_3T_{b-c}\,,\quad W_{d+c}=f_1T_{c-a}+f_2T_{c-b}\,.
\]
The identity of Theorem \ref{thm.ash3b4f} is established by solving these equations for $f_1$, $f_2$ and $f_3$ and substituting the solutions into identity \eqref{eq.nbl33qg}.

\end{proof}
\begin{corollary}\label{cor.e168rzg}
The following identities hold for integers $r$ and $s$:
\begin{equation}\label{eq.kr1imzv}
4W_{r + s}  = 2T_{s - 1} W_{r - 4}  + (T_{s + 4}  - 7T_s )W_r  + 4T_s W_{r + 1}\,,
\end{equation}
\begin{equation}\label{eq.yv19j1l}
4W_{r + s}  = 2T_{s - 4} W_{r - 1}  + (4T_{s + 1}  - 7T_s )W_r  + T_s W_{r + 4}\,,
\end{equation}
\begin{equation}\label{eq.ul2k2pm}
W_{r + s}  = T_{s - 1} W_{r - 1}  + (T_{s + 1}  - T_s )W_r  + T_s W_{r + 1}\,,
\end{equation}
\begin{equation}\label{eq.lh1k3sx}
4W_{r + s}  = T_{s - 4} W_{r - 4}  + (T_{s + 4}  - 11T_s )W_r  + T_s W_{r + 4}\,,
\end{equation}
\begin{equation}\label{eq.ebn9kfg}
W_{r + s}  = (T_{s + 1}  - 2T_s  - T_{s - 2} )W_{r - 1}  + (T_{s + 1}  - 2T_s )W_r  + T_s W_{r + 2}\,.
\end{equation}

\end{corollary}
\begin{proof}
See Table \ref{tabl.ijmi0p1}.
\begin{table}[h!]
\begin{tabular}{lcccccllllllll}
\multicolumn{1}{c}{Identity} & $a$ & $b$ & $c$ & $d$ & $e$ & \multicolumn{1}{c}{} & \multicolumn{1}{c}{} & \multicolumn{1}{c}{Identity} & \multicolumn{1}{c}{$a$} & \multicolumn{1}{c}{$b$} & \multicolumn{1}{c}{$c$} & \multicolumn{1}{c}{$d$} & \multicolumn{1}{c}{$e$} \\ 
\cline{1-6}\cline{9-14}
\multicolumn{1}{c}{\eqref{eq.kr1imzv}} & $r-4$ & $r$ & $r+1$ & $0$ & $r+s$ & \multicolumn{1}{c}{} & \multicolumn{1}{c}{} & \multicolumn{1}{c}{\eqref{eq.lh1k3sx}} & \multicolumn{1}{c}{$r+4$} & \multicolumn{1}{c}{$r$} & \multicolumn{1}{c}{$r-4$} & \multicolumn{1}{c}{$0$} & \multicolumn{1}{c}{$r+s$} \\ 
\cline{1-6}\cline{9-14}
\multicolumn{1}{c}{\eqref{eq.yv19j1l}} & $r+4$ & $r$ & $r-1$ & $0$ & $r+s$ & \multicolumn{1}{c}{} & \multicolumn{1}{c}{} & \multicolumn{1}{c}{\eqref{eq.ebn9kfg}} & \multicolumn{1}{c}{$r+2$} & \multicolumn{1}{c}{$r-1$} & \multicolumn{1}{c}{$r$} & \multicolumn{1}{c}{$0$} & \multicolumn{1}{c}{$r+s$} \\ 
\cline{1-6}\cline{9-14}
\multicolumn{1}{c}{\eqref{eq.ul2k2pm}} & $r-1$ & $r$ & $r+1$ & $0$ & $r+s$ & \multicolumn{1}{c}{} & \multicolumn{1}{c}{} & \multicolumn{1}{c}{} & \multicolumn{1}{c}{} & \multicolumn{1}{c}{} & \multicolumn{1}{c}{} & \multicolumn{1}{c}{} & \multicolumn{1}{c}{} \\ 
\cline{1-6}
\end{tabular}
\caption{Appropriate substitutions in the identity of Theorem \ref{thm.ash3b4f} to obtain \mbox{identities \eqref{eq.kr1imzv} -- \eqref{eq.ebn9kfg}} of Corollary \ref{cor.e168rzg}.}
\label{tabl.ijmi0p1}
\end{table}

\end{proof}
Note that identity \eqref{eq.ul2k2pm} can be written in the familiar form
\begin{equation}\label{eq.d86jm8d}
W_{r + s}  = T_{s - 1} W_{r - 1}  + (T_{s - 1}  + T_{s-2} )W_r  + T_s W_{r + 1}\,.
\end{equation}
Evaluating identity \eqref{eq.kr1imzv} at $s=-3$, $s=-16$ and at $s=-17$, in turn, we find the following three term recurrence relations for the generalized Tribonacci numbers:
\begin{equation}\label{eq.df8s42u}
W_{r-3}=2W_r-W_{r+1}\,,
\end{equation}
\begin{equation}\label{eq.xnvfyzv}
W_{r-16}=-103W_r+56W_{r+1}
\end{equation}
and
\begin{equation}\label{eq.zj12pm9}
2W_{r-17}=9W_r-103W_{r-4}\,.
\end{equation}
Evaluating identity \eqref{eq.ebn9kfg} at $s=-14$ produces yet another three-term recurrence relation for the generalized Tribonacci numbers:
\begin{equation}\label{eq.kefb5d6}
W_{r-14}=9W_{r+2}-56W_{r-1}\,.
\end{equation}
Note that by interchanging $r$ and $s$ in each case and making use of the defining recurrence relation for $W$ and $T$, the identities \eqref{eq.kr1imzv} -- \eqref{eq.ebn9kfg} can also be written 
\begin{equation}
4W_{r + s}  = 2W_{s - 1} T_{r - 4}  + (W_{s + 4}  - 7W_s )T_r  + 4W_s T_{r + 1}\,,
\end{equation}
\begin{equation}
4W_{r + s}  = 2W_{s - 4} T_{r - 1}  + (4W_{s + 1}  - 7W_s )T_r  + W_s T_{r + 4}\,,
\end{equation}
\begin{equation}\label{eq.j6edvhd}
W_{r + s}  = W_{s - 1} T_{r - 1}  + (W_{s + 1}  - W_s )T_r  + W_s T_{r + 1}\,,
\end{equation}
\begin{equation}
4W_{r + s}  = W_{s - 4} T_{r - 4}  + (W_{s + 4}  - 11W_s )T_r  + W_s T_{r + 4}\,,
\end{equation}
\begin{equation}
W_{r + s}  = (W_{s + 1}  - 2W_s  - W_{s - 2} )T_{r - 1}  + (W_{s + 1}  - 2W_s )T_r  + W_s T_{r + 2}\,.
\end{equation}
Evaluating identity \eqref{eq.j6edvhd} at $s=-2$ with $W=K$ gives
\begin{equation}\label{eq.w931ajy}
K_{r-2}=5T_{r-1}-T_{r+1}\,,
\end{equation}
a three-term relation connecting the Tribonacci numbers and the Tribonacci-Lucas numbers.
\begin{theorem}\label{thm.wv4i8l5}
The following identity holds for integers $a$, $b$, $c$, $d$ and $e$:
\[
\begin{split}
\{K_{b-a-1}K_{a-b-1}+&K_{b-a-1}K_{c-b-1}K_{a-c-1}+K_{c-a-1}K_{b-c-1}K_{a-b-1}\\
&\qquad\qquad+K_{c-a-1}K_{a-c-1}+K_{c-b-1}K_{b-c-1}-1\}W_{d+e}\\
&=\{(1-K_{c-b-1}K_{b-c-1})K_{e-a}+(K_{b-a-1}+K_{c-a-1}K_{b-c-1})K_{e-b}\\
&\qquad\qquad+(K_{c-a-1}+K_{b-a-1}K_{c-b-1})K_{e-c}\}W_{d+a-1}\\
&\quad+\{(1-K_{c-a-1}K_{a-c-1})K_{e-b}+(K_{a-b-1}+K_{c-b-1}K_{a-c-1})K_{e-a}\\
&\qquad\qquad+(K_{c-b-1}+K_{c-a-1}K_{a-b-1})K_{e-c}\}W_{d+b-1}\\
&\qquad+\{(1-K_{b-a-1}K_{a-b-1})K_{e-c}+(K_{b-c-1}+K_{b-a-1}K_{a-c-1})K_{e-b}\\
&\qquad\qquad+(K_{a-c-1}+K_{b-c-1}K_{a-b-1})K_{e-a}\}W_{d+c-1}\,.
\end{split}
\]
\end{theorem}
\begin{proof}
We wish to express a generalized Tribonacci number as a linear combination of three Tribonacci-Lucas numbers. Let 
\begin{equation}\label{eq.a9nh9kp}
W_{d+e}=f_1K_{e-a}+f_2K_{e-b}+f_3K_{e-c}\,,
\end{equation}
where $a$, $b$, $c$, $d$ and $e$ are arbitrary integers and the coefficients $f_1$, $f_2$ and $f_3$ are to be determined. Setting $e=a-1$, $e=b-1$ and $e=c-1$, in turn, we obtain three simultaneous equations:
\[
\begin{split}
&W_{d+a-1}=-f_1+f_2K_{a-b-1}+f_3K_{a-c-1}\,,\quad W_{d+b-1}=f_1K_{b-a-1}-f_2+f_3K_{b-c-1}\,,\\
&\qquad W_{d+c-1}=f_1K_{c-a-1}+f_2K_{c-b-1}-f_3\,.
\end{split}
\]
The identity of Theorem \ref{thm.wv4i8l5} is obtained by solving these equations for $f_1$, $f_2$ and $f_3$ and substituting the solutions into identity \eqref{eq.a9nh9kp}.

\end{proof}
\begin{corollary}
The following identities hold for integers $r$ and $s$:
\begin{equation}\label{eq.cfph541}
\begin{split}
44W_{r+s} &= (9K_{s+3}-K_{s+5}+2K_{s+1})W_{r-6}\\
&\quad+(9K_{s+1}+K_{s+5}+2K_{s+3})W_{r-4}\\
&\qquad+(K_{s+3}+6K_{s+5}-K_{s+1})W_{r-2}\,,
\end{split}
\end{equation}
\begin{equation}
\begin{split}
88W_{r+s} &= (8K_{s+3}-30K_{s-1}-2K_{s-2})W_{r}\\
&\quad+(K_{s+3}-K_{s-1}+8K_{s-2})W_{r-4}\\
&\qquad+(3K_{s+3}+19K_{s-1}+2K_{s-2})W_{r+1}\,,
\end{split}
\end{equation}
\begin{equation}
\begin{split}
22W_{r+s} &= (5K_{s-2}+K_{s-1}+2K_{s})W_{r-1}\\
&\quad+(K_{s-2}-2K_{s-1}+7K_{s})W_{r}\\
&\qquad+(2K_{s-2}+7K_{s-1}+3K_{s})W_{r+1}\,,
\end{split}
\end{equation}
\begin{equation}
\begin{split}
5060W_{r+s} &= (-264K_{s-3}+K_{s+9}-1079K_{s-1})W_{r-10}\\
&\quad+(-4279K_{s-3}+21K_{s+9}-21394K_{s-1})W_{r}\\
&\qquad+(6325K_{s-1}+1265K_{s-3})W_{r+2}
\end{split}
\end{equation}
and
\begin{equation}\label{eq.exe6rnf}
\begin{split}
5060W_{r+s} &= (-264K_{s-11}+1265K_{s+1}-4279K_{s-1})W_{r-2}\\
&\quad+(-1079K_{s-11}+6325K_{s+1}-21394K_{s-1})W_{r}\\
&\qquad+(21K_{s-1}+K_{s-11})W_{r+10}\,.
\end{split}
\end{equation}

\end{corollary}

Identities \eqref{eq.cfph541} -- \eqref{eq.exe6rnf} can also be written as follows.
\begin{equation}
\begin{split}
44W_{r+s} &= (9W_{s+3}-W_{s+5}+2W_{s+1})K_{r-6}\\
&\quad+(9W_{s+1}+W_{s+5}+2W_{s+3})K_{r-4}\\
&\qquad+(W_{s+3}+6W_{s+5}-W_{s+1})K_{r-2}\,,
\end{split}
\end{equation}
\begin{equation}
\begin{split}
88W_{r+s} &= (8W_{s+3}-30W_{s-1}-2W_{s-2})K_{r}\\
&\quad+(W_{s+3}-W_{s-1}+8W_{s-2})K_{r-4}\\
&\qquad+(3W_{s+3}+19W_{s-1}+2W_{s-2})K_{r+1}\,,
\end{split}
\end{equation}
\begin{equation}
\begin{split}
22W_{r+s} &= (5W_{s-2}+W_{s-1}+2W_{s})K_{r-1}\\
&\quad+(W_{s-2}-2W_{s-1}+7W_{s})K_{r}\\
&\qquad+(2W_{s-2}+7W_{s-1}+3W_{s})K_{r+1}\,,
\end{split}
\end{equation}
\begin{equation}
\begin{split}
5060W_{r+s} &= (-264W_{s-3}+W_{s+9}-1079W_{s-1})K_{r-10}\\
&\quad+(-4279W_{s-3}+21W_{s+9}-21394W_{s-1})K_{r}\\
&\qquad+(6325W_{s-1}+1265W_{s-3})K_{r+2}
\end{split}
\end{equation}
and
\begin{equation}
\begin{split}
5060W_{r+s} &= (-264W_{s-11}+1265W_{s+1}-4279W_{s-1})K_{r-2}\\
&\quad+(-1079W_{s-11}+6325W_{s+1}-21394W_{s-1})K_{r}\\
&\qquad+(21W_{s-1}+W_{s-11})K_{r+10}
\end{split}
\end{equation}

\section{Quadratic relations}
Our goal in this section is to derive expressions involving only pure squares of generalized Tribonacci numbers. To achieve this we must be able to express the anticipated cross-terms such as $W_{r - 1} W_r$ and $W_{r - 1} W_{r-4}$ as squares of generalized Tribonacci numbers.

\medskip

Rearranging identity~\eqref{eq.ccgxij2} and squaring, we have
\begin{equation}\label{eq.pbljs54}
4W_{r - 1} W_r  = 4W_{r - 1}^2  - W_{r - 4}^2  + W_r^2\,,
\end{equation}
\begin{equation}\label{eq.a41lrc2}
4W_{r - 1} W_{r - 4}  = 4W_{r - 1}^2  + W_{r - 4}^2  - W_r^2
\end{equation}
and
\begin{equation}\label{eq.qjbtgy8}
2W_r W_{r - 4}  = 4W_{r - 1}^2  - W_{r - 4}^2  - W_r^2\,.
\end{equation}
Rearranging identity~\eqref{eq.ccgxij2} and multiplying through by $4W_{r-3}$ to obtain
\begin{equation}
8W_{r - 1} W_{r - 3}  = 4W_r W_{r - 3}  + 4W_{r - 4} W_{r - 3}\,,
\end{equation}
and using identities~\eqref{eq.pbljs54} and~\eqref{eq.a41lrc2} to resolve the right hand side gives
\begin{equation}\label{eq.y9lf0gg}
8W_{r - 1} W_{r - 3}  = 4W_r^2  + 2W_{r - 3}^2  - W_{r + 1}^2  + 4W_{r - 4}^2  - W_{r - 7}^2\,.
\end{equation}
Multiplying through identity~\eqref{eq.ccgxij2} by $4W_{r-5}$ to obtain
\begin{equation}
4W_{r} W_{r - 5}  = 8W_{r-1} W_{r - 5}  - 4W_{r - 4} W_{r - 5}\,,
\end{equation}
which, with the use of \eqref{eq.pbljs54} and \eqref{eq.qjbtgy8}, translates to
\begin{equation}\label{eq.gsxw34t}
4W_{r}W_{r-5} = 16W_{r-2}^2-8W_{r-5}^2-4W_{r-1}^2+W_{r-8}^2-W_{r-4}^2\,.
\end{equation}
Rearranging identity~\eqref{eq.ccgxij2}, shifting index $r$ and multiplying through by $2W_{r}$ gives
\begin{equation}
2W_{r} W_{r - 8}  = 4W_r W_{r - 5}  - 2W_{r} W_{r - 4}\,,
\end{equation}
from which, using \eqref{eq.qjbtgy8} and \eqref{eq.gsxw34t}, we get
\begin{equation}\label{eq.q2bp8ae}
2W_{r}W_{r-8} = 16W_{r-2}^2-8W_{r-5}^2-8W_{r-1}^2+W_{r-8}^2+W_{r}^2\,.
\end{equation}

Re-arranging identity \eqref{eq.xnvfyzv} and squaring, we have
\begin{equation}\label{eq.m5pbty8}
112W_{r-17}W_{r} = W_{r-17}^2+3136W_{r}^2-10609W_{r-1}^2\,
\end{equation}
\begin{equation}\label{eq.wos5d03}
11536W_{r-1}W_{r} = -W_{r-17}^2+3136W_{r}^2+10609W_{r-1}^2\,,
\end{equation}
and
\begin{equation}\label{eq.rxocfxa}
206W_{r-17}W_{r-1} = -W_{r-17}^2-10609W_{r-1}^2+3136W_{r}^2\,.
\end{equation}
Rearranging and squaring identity \eqref{eq.zj12pm9}, we have
\begin{equation}\label{eq.cm0gepi}
36W_{r-17}W_{r} = 4W_{r-17}^2+81W_{r}^2-10609W_{r-4}^2\,,
\end{equation}
\begin{equation}\label{eq.ebfmewt}
1854W_{r}W_{r-4} = 81W_{r}^2+10609W_{r-4}^2-4W_{r-17}^2
\end{equation}
and
\begin{equation}\label{eq.q7trtk8}
412W_{r-17}W_{r-4} = -4W_{r-17}^2-10609W_{r-4}^2+81W_{r}^2\,.
\end{equation}
Finally, squaring and re-arranging identity \eqref{eq.kefb5d6} produces
\begin{equation}\label{eq.upegzaa}
18W_{r-17}W_{r-1} = W_{r-17}^2+81W_{r-1}^2-3136W_{r-4}^2\,,
\end{equation}
\begin{equation}
1008W_{r-1}W_{r-4} = 81W_{r-1}^2+3136W_{r-4}^2-W_{r-17}^2
\end{equation}
and
\begin{equation}\label{eq.abtrmog}
112W_{r-17}W_{r-4} = -W_{r-17}^2-3136W_{r-4}^2+81W_{r-1}^2\,.
\end{equation}
\begin{theorem}\label{thm.ldkszku}
The following identities hold for $r$ an integer:
\begin{equation}\label{eq.p1qgtxd}
252W_r^2  - 927W_{r - 1}^2 + 2884W_{r - 4}^2 - W_{r - 17}^2= 0\,,
\end{equation}
\begin{equation}\label{eq.jesyg5e}
W_r^2  - 2W{}_{r - 1}^2  - 3W_{r - 2}^2  - 6W_{r - 3}^2  + W_{r - 4}^2  + W_{r - 6}^2  = 0\,.
\end{equation}
\end{theorem}
\begin{proof}
Eliminating $W_{r-1}W_r$ between identities \eqref{eq.pbljs54} and \eqref{eq.wos5d03} proves identity \eqref{eq.p1qgtxd}.
To prove identity~\eqref{eq.jesyg5e}, write $W_r-W_{r-1}=W_{r-2}+W_{r-3}$, square both sides and use the identity~\eqref{eq.pbljs54} to resolve the cross-products $W_r W_{r - 1} $ and $W_{r-2} W_{r - 3} $.

\end{proof}
Substituting for $W_{r-17}^2$ from the identity of Theorem \ref{thm.ldkszku} into identities \eqref{eq.m5pbty8}, \eqref{eq.rxocfxa} and \eqref{eq.abtrmog}, we have the following simpler versions of these identities:
\begin{equation}
4W_{r-17}W_{r} = -412W_{r-1}^2+103W_{r-4}^2+121W_{r}^2\,,
\end{equation}
\begin{equation}
W_{r-17}W_{r-1} = -47W_{r-1}^2-14W_{r-4}^2+14W_{r}^2
\end{equation}
and
\begin{equation}
4W_{r-17}W_{r-4} = 36W_{r-1}^2-215W_{r-4}^2-9W_{r}^2\,.
\end{equation}
Next we give how to express the square of a Tribonacci-Lucas number in terms of squares of Tribonacci numbers.
\begin{theorem}
The following identity holds for any integer $r$:
\begin{equation}
4K_{r}^2 =5T_{r+5}^2-20T_{r+4}^2+4T_{r+3}^2+90T_{r+1}^2-20T_{r}^2+5T_{r-3}^2\,.
\end{equation}
\end{theorem}
\begin{proof}
Square identity \eqref{eq.w931ajy} and use identity \eqref{eq.y9lf0gg} to eliminate the cross term.
\end{proof}
\begin{theorem}
The following identities hold for $r$ and $s$ integers:
\begin{equation}
\begin{split}
16W_{r+s}^2 &= -(T_{s}+T_{s+4})(-T_{s+4}+2T_{s-1}+7T_{s})W_{r}^2+4T_{s-1}T_{s}W_{r-7}^2\\
&\quad+2T_{s-1}(-T_{s+4}-9T_{s}+2T_{s-1})W_{r-4}^2-2T_{s}(T_{s+4}-7T_{s}+2T_{s-1})W_{r-3}^2\\
&\qquad+8T_{s-1}(T_{s}+T_{s+4})W_{r-1}^2+2T_{s}(T_{s}+T_{s+4})W_{r+1}^2\,,
\end{split}
\end{equation}
\begin{equation}
\begin{split}
16W_{r+s}^2 &= 4(2T_{s}-T_{s+1})(7T_{s}-T_{s-4}-4T_{s+1})W_{r}^2+4T_{s-4}(2T_{s}-T_{s+1})W_{r-4}^2\\
&\quad-4T_{s-4}(-T_{s-4}-4T_{s+1}+9T_{s})W_{r-1}^2+16T_{s-4}T_{s}W_{r+2}^2\\
&\qquad-4T_{s}(7T_{s}+T_{s-4}-4T_{s+1})W_{r+3}^2+4T_{s}(2T_{s}-T_{s+1})W_{r+4}^2\,,
\end{split}
\end{equation}
\begin{equation}
\begin{split}
4W_{r+s}^2 &= -2(-T_{s+1}+T_{s})(2T_{s+2}-T_{s-1})W_{r}^2-T_{s-1}T_{s}W_{r-5}^2\\
&\quad+2T_{s-1}(-T_{s+1}+T_{s})W_{r-4}^2+2T_{s}(-T_{s+1}+T_{s})W_{r-3}^2\\
&\qquad+4T_{s-1}T_{s}W_{r-2}^2+2T_{s-1}(2T_{s-1}+4T_{s+1}-3T_{s})W_{r-1}^2\\
&\quad\qquad+2T_{s}(T_{s}+T_{s+1})W_{r+1}^2-T_{s-1}T_{s}W_{r+3}^2+4T_{s-1}T_{s}W_{r+2}^2\,.
\end{split}
\end{equation}

\end{theorem}
\section{Cubic recurrence relations}
\begin{theorem}\label{thm.efw58sf}
The following identity holds for integer $r$:
\[
\begin{split}
&W_r^3  - 4W_{r - 1}^3  - 9W_{r - 2}^3  - 34W_{r - 3}^3  + 24W_{r - 4}^3  - 2W_{r - 5}^3\\
&\quad + 40W_{r - 6}^3  - 14W_{r - 7}^3  - W_{r - 8}^3  - 2W_{r - 9}^3  + W_{r - 10}^3  = 0\,.
\end{split}
\]
\end{theorem}
\begin{proof}
Setting $a=r-8$, $b=r$, $c=r-10$, $d=0$ and $e=r-s$ with $s\in\{1,2,3,4,5,6,7,9\}$ in the identity of Theorem \ref{thm.ash3b4f}, the following linear combinations are formed:
\[
W_{r - 1}  = \frac{{37}}{{68}}W_r  - \frac{5}{{68}}W_{r - 8}  + \frac{1}{{17}}W_{r - 10} ,\quad W_{r - 2}  = \frac{5}{{17}}W_r  + \frac{3}{{17}}W_{r - 8}  + \frac{1}{{17}}W_{r - 10}\,, 
\]
\[
W_{r - 3}  = \frac{{11}}{{68}}W_r  - \frac{7}{{68}}W_{r - 8}  - \frac{2}{{17}}W_{r - 10} ,\quad W_{r - 4}  = \frac{3}{{34}}W_r  - \frac{5}{{34}}W_{r - 8}  + \frac{2}{{17}}W_{r - 10}\,,
\]
\[
W_{r - 5}  = \frac{3}{{68}}W_r  + \frac{{29}}{{68}}W_{r - 8}  + \frac{1}{{17}}W_{r - 10} ,\quad W_{r - 6}  = \frac{1}{{34}}W_r  - \frac{{13}}{{34}}W_{r - 8}  - \frac{5}{{17}}W_{r - 10}\,,
\]
\[
W_{r - 7}  = \frac{1}{{68}}W_r  - \frac{{13}}{{68}}W_{r - 8}  + \frac{6}{{17}}W_{r - 10} ,\quad W_{r - 9}  = \frac{1}{{68}}W_r  - \frac{{81}}{{68}}W_{r - 8}  - \frac{{11}}{{17}}W_{r - 10}\,. 
\]
The above $W_{r-i},i\in\{1,2,3,4,5,6,7,9\}$ verify the identity of Theorem \ref{thm.efw58sf}.
\end{proof}
Taking the cube in identities \eqref{eq.df8s42u} and \eqref{eq.xnvfyzv} and solving two simultaneous equations, we find
\begin{equation}\label{eq.q4ahbd9}
155736W_{r}W_{r-1}^2 = -W_{r-17}^3+199305W_{r-1}^3-161504W_{r-4}^3+14112W_{r}^3
\end{equation}
and
\begin{equation}\label{eq.zs0hzum}
77868W_{r}^2W_{r-1} = -148526W_{r-4}^3+27090W_{r}^3-W_{r-17}^3+95481W_{r-1}^3\,.
\end{equation}
\begin{theorem}
The following identity holds for integer $r$:
\begin{equation}
\begin{split}
&11844W_{r}^3+3W_{r-19}^3+W_{r-17}^3+458556W_{r-6}^3\\
&\quad+135548W_{r-4}^3-442179W_{r-3}^3-120204W_{r-2}^3-43569W_{r-1}^3=0\,.
\end{split}
\end{equation}
\end{theorem}
\begin{proof}
Write the defining recurrence relation of the generalized Tribonacci numbers as $W_r-W_{r-1}=W_{r-2}+W_{r-3}$, take the cube of both sides and use identities \eqref{eq.q4ahbd9} and \eqref{eq.zs0hzum} to remove cross terms.
\end{proof}

\hrule

\noindent 2010 {\it Mathematics Subject Classification}:
Primary 11B39; Secondary 11B37.

\noindent \emph{Keywords: }
Tribonacci number, Tribonacci-Lucas number, recurrence relation.

\hrule


\begin{thebibliography}{99}

\bibitem{adegoke18c} K.~Adegoke, Weighted Tribonacci sums, \emph{arXiv:1804.06449[math.CA]} (2018).

\bibitem{feng11} J.~Feng, More identities on the Tribonacci numbers, \emph{Ars Combinatorial} {\bf C} (2011), 73--78.

\bibitem{frontczak18} R.~Frontczak, Sums of Tribonacci and Tribonacci-Lucas numbers, \emph{International Journal of Mathematical Analysis} {\bf 12}:1 (2018), 19--24.

\bibitem{gabai70} H. Gabai, Generalized Fibonacci  $k$-sequences, \emph{The Fibonacci Quarterly} {\bf 8}:1 (1970), 31--38.

\bibitem{shah11} D.~V.~Shah, Some Tribonacci identities, \emph{Mathematics Today} {\bf 27} (2011), 1--9.

\bibitem{waddill67} M. E. Waddill and L. Sacks, Another generalized Fibonacci sequence, \emph{The Fibonacci Quarterly} {\bf 5}:3 (1967), 209--222.

\end{thebibliography}
\end{document}